\documentclass[a4paper,12pt]{amsart}
\usepackage{enumerate}
\usepackage{amssymb,latexsym,amsxtra,amscd,amsfonts}
\usepackage[mathscr]{eucal}
\usepackage{dsfont}
\usepackage{calrsfs}

\newcommand{\thmref}[1]{Theorem~\ref{#1}}

\newcommand{\lemref}[1]{Lemma~\ref{#1}}
\newcommand{\prref}[1]{Proposition~\ref{#1}}
\newcommand{\corref}[1]{Corollary~\ref{#1}}

\makeatletter
\def\@biblabel#1{#1.}
\makeatother
\newcounter{obr}[section]

\newcounter{pvv}[section]
\renewcommand{\thepvv}{\thesection.\arabic{pvv}}
\newenvironment{pv}[2][]{\begin{trivlist}\refstepcounter{pvv}%
\item[\hspace{\labelsep}\normalfont\bfseries\thepvv. #2%
  \def\tmp{#1}\ifx\tmp\empty\else{} (#1)\fi.]}%
{\end{trivlist}}
\newenvironment{df}{\begin{pv}{Definition}}{\end{pv}}
\newenvironment{theo}[1][]{\begin{pv}[#1]{Theorem}\begin{itshape}}{\end{itshape}\end{pv}}
\newenvironment{pr}{\begin{pv}{Proposition}\begin{itshape}}{\end{itshape}\end{pv}}
\newenvironment{lem}{\begin{pv}{Lemma}\begin{itshape}}{\end{itshape}\end{pv}}

\newenvironment{cor}{\begin{pv}{Corollary}\begin{itshape}}{\end{itshape}\end{pv}}

\newcommand{\set}[2]{\ensuremath{\{ #1\,:\; #2\}}}

\newcommand{\nn}{\ensuremath{\mathbb N}}

\newcommand{\cc}{\ensuremath{\mathbb C}}

\newcommand{\8}{\ensuremath{\infty}}

\newcommand{\f}{\ensuremath{\varphi}}
 
\renewcommand{\l}{\ensuremath{\lambda}}
\newcommand{\ro}{\ensuremath{\varrho}}

\newcommand{\un}{\ensuremath{{\bf 1}}} 

\newcommand{\al}{\ensuremath{\alpha}}

\newcommand{\ifff}{if, and only if,}
\newcommand{\vNa}{von Neumann algebra}

\newcommand{\Ca}{$C^\ast$-algebra}
\newcommand{\Csa}{$C^\ast$-subalgebra}

\newcommand{\HS}{Hilbert space}

\newcommand{\st}{such that}

\newcommand{\seq}[2]{\ensuremath{#1_1, \ldots, #1_{#2}}}

\newcommand{\kol}[1]{\ensuremath{#1^{\perp}}}

\newcommand{\nor}[1]{\ensuremath{\|#1\|}}

\newcommand{\AW}{$AW^\ast$-algebra}
\newcommand{\GT}{Gleason's Theorem}
\newcommand{\WT}{Wigner's Theorem}
\newcommand{\DT}{Dye's Theorem}

\begin{document}
 \begin{center}
 {\bf \Large \DT{} and \GT{} for \AW s. } \\[2cm]

 Jan Hamhalter \\
 Czech Technical University\\
 Faculty of Electrical Engineering\\
 Technicka 2, 166 27, Prague 6\\
 hamhalte@math.feld.cvut.cz
 \end{center}

 \vspace{5mm}
 
 {\small Abstract: We prove that any map between projection lattices of \AW s $A$ and $B$, where $A$ has no Type $I_2$ direct summand, that preserves orthocomplementation and suprema of arbitrary elements,  is a restriction of a normal Jordan $\ast$-homomorphism between $A$ and $B$. This allows us to generalize \DT{}  from \vNa s to \AW s. We show that Mackey-Gleason-Bunce-Wright Theorem  can be extended to homogeneous \AW s of Type I. The interplay between \DT{} and \GT{} is shown. As an application we prove  that Jordan $\ast$-homomorphims  are commutatively determined. Another corollary says that Jordan parts of \AW s can be reconstructed from  posets of their abelian subalgebras.   }

 \section{Introduction}

   The main goal of the present  paper is to show that any map between projection lattices 
   of \AW s that preserves orthocomplementation and arbitrary suprema is
 a restriction of  a normal Jordan $\ast$-homomorphism. This generalizes famous Dye's Theorem in a few directions. Moreover, we contribute to the  Mackey-Gleason problem by showing that any bounded vector measure on the  projection lattice of \AW{} of finite Type $I_n$, 
 ($n\ge 3$), extends to a bounded linear map.     
     Besides its own mathematical interest, this line of the research   stems also from a long  discussion on mathematical understanding of    quantum theory. There are two basic principles of mathematical foundations of quantum mechanics   - Gleason's Theorem and Wigner's Theorem. These results are very nontrivial,  even in the  the most special context of matrix algebras (quantum systems with finitely many  levels). In this setting they read as follows. Gleason's Theorem states that any probability measure  on the projection structure, $P(M_n(\cc))$,  of the matrix algebra $M_n(\cc)$, $n\ge 3$, of all complex $n$ by $n$ matrices,  extends to
a positive linear functional on $M_n(\cc)$ \cite{Gleason}. Loosely   speaking, it says that any quantum   probability measure has its expectation value (integral). On the other hand, Wigner's Theorem \cite{Wigner} (in its Ulhron's version \cite{Ulhron}) says that any bijection \f{} acting on $P(M_n(\cc))$, $n\ge 3$, that preserves orthogonality in both directions  is implemented by a unitary or anti-unitary operator $U$ in the sense 
\[ \f(P)= U^{-1}PU\,, \qquad P\in P(M_n(\cc))\,.\]
It means  that any transformation that preserves logical structure of the system of quantum propositions is in fact a geometric transformation of the underlying inner product space. 
 \GT{} and \WT{} have one thing in common. They show  that a particular mathematical object attached to projections can be viewed globally as attached to the linear structure of the whole space or algebra. Responding to a  well known Mackey's problem in axiomatic of quantum theory \cite{Mackey},
 Gleason showed in \cite{Gleason} that any probability measure on the structure $P(H)$ of projections  that act on a separable \HS{} $H$,  extends to a positive functional on the algebra $B(H)$ of all bounded operators acting on $H$. After considerable effort of many outstanding  mathematicians 
 \cite{Aarnes,Christensen,Pasz,Yeadon1,Yeadon2}, \GT{} has been extended to positive (finitely additive) measures on projection lattices of general von Neumann algebras (see also survey \cite{qmt,Maeda} and references therein). However, in order to obtain \GT{} for  vector measures it was necessary to relax positivity assumption and to prove the result for general complex-valued measures. 
This progress   required further difficult  ideas and techniques. It was achieved  in a remarkable series of works by Bunce and Wright \cite{Bunce1,Bunce2,Bunce4}.  This development has led to the
Mackey-Gleason-Bunce-Wright Theorem:  Let $M$ be a \vNa{} without Type $I_2$ direct summand
with projection lattice $P(M)$. Any bounded  map $\ro:P(M)\to X$ with values in a Banach space $X$ \st{}
\[ \ro(p+q)=\ro(p)+\ro(q)\,,\]
whenever $p$ and $q$ are orthogonal projections, extends to a bounded operator $T:M\to X$. 

Despite the progress in establishing linear extensions of measures for \vNa s, not much is known about solution of Mackey-Gleason problem for general \Ca s. One of the most known results in this direction is Haagerup's theorem saying that any quasi trace on exact \Ca{} is linear \cite{Haagerup}.
\\

   \WT{} has an intriguing history as well. The main role has been played by the following reformulation of this result:  
  Any orthoisomorphism of the  orthomodular lattice  $P(H)$, $\dim H\ge 3$, extends to 
  a Jordan $\ast$-isomorphism of the algebra $B(H)$. Remarkable \DT{} \cite{Dye} extends this to a very general context of  \vNa s: Let $M$ ad $N$ be \vNa s with projection lattices 
  $P(M)$ and $P(N)$, respectively. If $M$ has no Type $I_2$ direct summand, then any orthoisomorphism $\f:P(M)\to P(N)$ extends to a unique Jordan $\ast$-isomorphism $\Phi:M\to N$. \DT{} is one of the  deepest results on the geometry of projections in \vNa s. 
 The arguments in the proof of \DT{} rely on geometry of matricidal structures over \vNa s and on applying special  lattice polynomials that, surprisingly, have the power to capture  linear structure. Some of these ideas have their origin in von Neumann work on projective geometry \cite{vN-projective geometry}.  In the proof of \DT{} the bijectivity of orthoisomorphism is used in an essential way.\\

       The proofs of \GT{} and \DT{} were independent for a long time.  However, a clever argument was given by Bunce and Wright \cite{Bunce3} to the effect that \GT{} for positive measures on \vNa s  implies quickly \DT{}. Moreover, it was shown  in \cite{Bunce3} that any map $\f:P(M)\to P(N)$ between projection lattices of \vNa s $M$ and $N$, where $M$ has no  Type $I_2$ direct summand,
    is a restriction of Jordan $\ast$-homomorphism between $M$ and $N$ if and only if 
  $\f(p+q)=\f(p)+\f(q)$ for any  pair of orthogonal projections $p$ and $q$ in $M$.  
Recently, the problem of linear extensions of projection lattice morphisms has been investigated in the context of \AW s by Heunen and Reyes \cite{Active}. They succeeded in proving the following deep result. Any map $\f: P(A)\to P(B)$ between projection lattices of \AW s
$A$ and $B$, where $A$ has no  Type $I_2$ direct summand, that preserves 
arbitrary suprema and orthocomplements,  extends to a normal Jordan $\ast$-homomorphism between $A$ and $B$ if and only if the following equivariance condition holds:
\begin{equation}\label{equivar}
 \f((\un-2p)q(\un-2p))= (\un-2\f(p))\f(q)(\un-2\f(q))\,.
 \end{equation} 
 The authors posed the following open problem in \cite{Active}:
Does any 
morphism \f{}  specified above satisfy condition (\ref{equivar}) automatically? 
 We answer this problem in the positive.\\

 In order to obtain \DT{} for \AW s, one might be  tempted to follow ideas of Bunce and Wright and to try to establish the \GT{} for \AW s first. But this way seems to be very hard.
 For example, if the Mackey-Gleason problem has positive solution for \AW s, then any dimension function  on Type $II_1$ $AW^\ast$-factor extends to a trace. But the existence of such a trace would imply that any factorial \AW{} of Type $II_1$ is a von Neumann factor. This would solve difficult  
  Kaplansky's problem \cite{Kaplan}.
 However, fortunately, condition (\ref{equivar}) involves only two projections. We shall carefully examine 
 the structure of $AW^\ast$-subalgebras   generated by  two projections, and establish Mackey-Gleason-Bunce-Wright Theorem for this case. This implies  that  (\ref{equivar})  holds for any  map between projection lattices of \AW s (not having Type $I_2$ direct summand) that preserves arbitrary suprema of projections  and orthocomplements. Our main theorem then reads as follows. {\em Let $A$ be an \AW{} without Type $I_2$ direct summand, $B$ be an \AW{}, and  let $\f: P(A)\to P(B)$ be a map between projection lattices that preserves arbitrary suprema and orthocomplements. Then \f{} is the restriction of a normal Jordan $\ast$-homomorphism $\Phi: A\to B$.} Moreover, this result allows us to show that normal Jordan $\ast$-homomorphisms between  $AW^\ast$-algebras are commutatively determined, that is, a map (linear or not) is a a normal Jordan $\ast$-homomorphism if it is a normal Jordan $\ast$-homomorphism when restricted to  any abelian subalgebra. Besides,  we establish Mackey-Gleason-Bunce-Wright Theorem for \AW s of Type $I_n$, $3\le n<\8$,  as well. (The case of properly infinite algebras will be treated in a subsequent paper.) \\
 
   Let us remark that \AW s seem to be  more natural  for Mackey-Gleason program as well as for  logical considerations  on  quantum theory than \vNa s. For example, any quasitrace on a \Ca{} is a composition of a $\ast$-homomorphism and a quasitrace on a finite \AW{} \cite{Blackadar}.  Therefore, \AW s play a key role in linearity problem for quasi traces. On the other hand, only in the category of    \AW s we have a perfect bijective correspondence between commutative \AW s and complete Boolean algebras. In this correspondence \AW{} is sent to its projection lattice. This underlines crucial role  of \AW s for logical structures.   Finally, \AW s are more suitable for recent topos theoretic approach to 
   quantum theory (see e.g. \cite{Deep}).  This approach is based on the    structure $Abel(A)$ of commutative $C^\ast$-subalgebras of a given \Ca{}  $A$, ordered by the set inclusion. It was shown in \cite{Doring,Ham-Abel} that $Abel(A)$ determines the Jordan structure of a von Neumann algebra $A$. As an application of \DT{} for \AW s, we show that the same holds for \AW s: {\em Let $A$ be an \AW{} without Type $I_2$ direct summand and $B$ be any \AW.  Suppose that $\f:Abel(A)\to Abel(B)$   
   is an order isomorphism.  Then there is a unique Jordan $\ast$-isomorphism $\Phi:A\to B$ \st{}
   \[ \f(C)= \Phi(C)\,,\qquad  C\in Abel(A)\,.\]}
   This generalizes hitherto known results in \cite{Ham-Abel,Ham-ET}. \\

   The paper is organized as follows. In the second section we deal with the the geometry of the structure of projections in \AW s, especially we describe  \AW s generated by two projections and analyze  isoclinicity of projections. Inclusions of two by two matricidal substructures into \AW s is 
   examinated. In Section~3 we establish \GT{} for finite homogeneous 
   \AW s. This enables us to show linearity of quasi linear functionals on subalgebras generated by two projections. In the concluding section main results described above are presented.\\
   
   Let us now recall basic notions and fix the notation.
   For all unmentioned details on operator algebras we refer the reader to
  monographs \cite{Berberian,Kadison-Ringrose}. 
For a \Ca{} $A$ we shall denote by $\un$ its unit (if it exists). By $A_{sa}$ we shall understand  the real subspace of $A$ consisting of all self-adjoint elements. 
We write $A^+$ and $A_1$ for the positive part  and the closed unit ball of $A$, respectively.
By $P(A)$ we shall denote the set of all projections in $A$; that is the set of all self-adjoint idempotents. $P(A)$ is ordered by order relation $e\le f$ if $ef=e$. Suprema and infima of two projections $e$ and $f$ will be denoted by  
$e\lor f$ and $e\land f$, respectively (if they exist). An orthocomplement
$\kol{e}$ of a projection $e$ is defined as   $\kol{e}=\un -e$. 
The {\em central cover}, $c(e)$, of a projection $e$ is a smallest central projection $z$ for which $z\ge e$. We say that two projections are {\em very orthogonal} if their central covers are orthogonal.
A projection is called {\em faithful} if its central cover is the unit.
A projection $e\in A$ is called abelian if the hereditary  
subalgebra $eAe$ is abelian. 
 Finally,  two projections $e$ and $f$ are said to be {\em equivalent} (in symbols $e\sim f$) if there is an element $v$ (partial isometry) \st{} 
$vv^\ast=e$ and $v^\ast v =f$. By a {\em symmetry} we mean a self-adjoint element $s$ with $s^2=1$.  
Given a \Ca{} $A$, we shall write $M_n(A)$ for the \Ca{} of all $n$ by $n$ matrices with entries from $A$. If $X$ is a compact Hausdorff space, then by $C(X,A)$ we denote the \Ca{} of all continuous maps from $X$ to $A$. If $A=\cc$, we write simply $C(X)$.

   A Jordan $\ast$-homomorphism $\Phi$ is a linear map between two \Ca s that preserves $\ast$ operation and squares of self-adjoint elements. Jordan $\ast$-isomorphism is a Jordan $\ast$-homomorphism that is bijective and whose inverse is also a Jordan $\ast$-homomorphism. 
 An {\em \AW{}} is a \Ca{} $A$ that is a {\em Bear $\ast$-ring}. That is,  if the following holds: For any nonempty subset $S\subset A$ there is a   projection $e\in A$ such that for the right annihilator
 $R(S)=\set{a\in A}{sa=0\mbox{ for all }s\in S}$ we have that  $R(S)=eA$. 
Throughout the paper $A$ will always represent \AW. 
Given an element $a\in A$ there exists a {\em left support projection} $LP(a)$ of $a$  which is a smallest projection $g\in A$ \st{} $ga=a$. Analogously, there is a {\em right support projection} $RP(a)$, that is a smallest projection $h$ with
$ah=h$. Explicitly, $R(\{a\})=(1-RP(a))A$. It is known that $P(A)$ is a complete lattice.  
Given elements $\seq an{}\in A$ we write $AW^\ast(\seq an)$ for the smallest $AW^\ast$-subalgebra of $A$ containing elements $\seq an$.

 \section{Geometry of projections in \AW s}

 \begin{df}
 Let $e$ and $f$ be projections in a \Ca{}  $C$. We say that $e$ and $f$ are isoclinic with angle $\al$, $0<\al<\frac \pi 2$, if
   \[  efe= \cos^2\al\,  e \quad {\rm and }\quad   fef = \cos^2\al\,  f\,.\]
    \end{df}

 The following Proposition gathers important facts about isoclinic projections. The proofs can be found in \cite[p. 129-130]{qmt} and \cite{Maeda}.\\

 \begin{pr} Let $e$ and $f$ be projections in a \Ca{} $C$ that are isoclinic with angle \al.
 Then the following statements are true:
 \begin{enumerate}
 \item \Ca{}  $C^\ast(e,f)$ generated by $e$ and $f$  is $\ast$-isomorphic to $M_2(\cc)$.
\item $e$ and $f$ are unitarily equivalent in $C^\ast(e,f)$. 
\item $\nor{e-f}=\sin \al.$
\end{enumerate}
\end{pr}

  We are going to analyze the position of two projections in a general \AW{} $A$. 
  Let us recall a few notions (see \cite{Berberian}). Two projections $e$ and $f$ are said to be in  position $p'$ if $e\land (\un -f)=(\un-e)\land f=0$.  
  Projections $e$ and $f$ that are in  position $p'$ are said to be in position $p$ if, moreover, $e\land f=(\un -e)\land (\un-f)=0$. Let us remark that $e$ and $f$ are in  position $p'$ \ifff{} $LP(ef)=e$ and $RP(ef)=f$. Further, $e$ and $f$ are in position $p$ \ifff{} $RP(ef-fe)=1$
  (see \cite{Berberian}).

  \begin{pr}\label{position}
Suppose that $e$ and $f$ in  $A$ are projections in  position $p$. Then
\[ RP(\kol e fe)=e\,.\]
\end{pr}

\begin{proof}
 Denote  $g=RP(x)$.  In other words,   
\[ R(\{\kol e fe\})=\kol g A\,.\]
As $\kol e A\subset R\{\kol e fe\}$ we infer that $e\ge g$. Put $z=e-g$.
The proof will be completed if we show that $z=0$. Suppose, for a contradiction, that $z\not= 0$. Using the fact that $\kol g\in R(\{\kol e fe\})$,
we have
\[ \kol e fe z= \kol e fe-\kol e feg=\kol e fe-\kol e fe=0\,.\]
Consequently, $z\in R(\{\kol e fe\})$. This, together with the inequality
$z\le e$, means  that  
\[ z\in R(\{\kol e f\})=(\un-RP(\kol e f))A=\kol fA\,.\]
Therefore, $z\le   \kol f\land e = 0\,.$
\end{proof}

\begin{pr}\label{two projections}
Let $e$ and $f$ be projections in $A$ in position $p$. Then $AW^\ast(e,f)$ 
lies in an $AW^\ast$-subalgebra isomorphic to $M_2(C)$, where $C$ is an abelian
\AW. Moreover, when identifying $M_2(C)$ with $C(X, M_2(\cc))$, where $X$ is the spectrum of $C$, we can arrange for 
\begin{equation}\label{newstar} e(x)= \left(  \begin{array}{cc}
  1 & 0  \\
     0  & 0
\end{array}
\right)
\mbox{ and } f(x)= \left(  \begin{array}{cc}
  a(x) & \sqrt{a(x)-a^2(x)}  \\
   \sqrt{a(x)-a^2(x)}    & 1-a(x)
\end{array}
\right)\,,
\end{equation}
where $a(x)$ is a continuous function on $X$ with values in $[0,1]$.
\end{pr}

\begin{proof}
Put $x=\kol efe.$ Then $x^\ast x= ef\kol efe\in eAe\,.$
Consider the  polar decomposition 
\[ x= uh\,,\]
where $h=(x^\ast x)^{1/2}\in eAe$, and $u$ is a partial isometry with
\[ u^\ast u = RP(x) \qquad uu^\ast=LP(x)\,.\]
As we know from \prref{position} 
\[ RP(x)=e \]
\[ LP(x)=RP(x^\ast)=RP(ef\kol e)=\kol e\,.\]
Therefore $u$ is a partial isometry with initial projection $e$ and final projection $\kol e$. By a standard argument  $u$ introduces a matrix unit  which organizes whole algebra $A$ as $M_2(eAe)$. Identifying $eAe$ with upper left corner 
of the corresponding matrix we can identify
\begin{equation}\label{troj} e=\left(  \begin{array}{cc}
  e & 0  \\
    0   & 0 
\end{array}
\right)  \qquad \kol e = \left(  \begin{array}{cc}
  0 & 0   \\
     0  & e 
\end{array}
\right)\qquad u= \left(  \begin{array}{cc}
  0 & 0  \\
    e   & 0 
\end{array}
\right)\,.
\end{equation} 
We shall find matrix representation of $f$. Suppose
\begin{equation}\label{ctverec} f = \left(  \begin{array}{cc}
  f_{11} & f_{12}  \\
   f_{21}    & f_{22} 
\end{array}
\right)\,.
\end{equation}
Then
\[ \left(  \begin{array}{cc}
  0 & 0  \\
   f_{21}    & 0
\end{array}
\right)= \kol e fe = uh = \left(  \begin{array}{cc}
  0 & 0  \\
    e   & 0 
\end{array}
\right)
\left(  \begin{array}{cc}
  h & 0  \\
     0  & 0 
\end{array}
\right)=
\left(  \begin{array}{cc}
  0 & 0   \\
     h  & 0 
\end{array}
\right)\,.
\]
It gives $f_{21}=f_{12}=h$. Expanding  the identity $f=f^2$ we obtain the following conditions:
\begin{eqnarray}\label{cross}
f_{11}^2+h^2&=&f_{11}\\
f_{11}h + hf_{22}&=&h\\
hf_{11}+f_{22}h &=& h\\
h^2+f_{22}^2 &=& f_{22}\,.
\end{eqnarray}
It implies that $f_{11}$ and $f_{22}$ commute with $h$ and so we have 
\[ h(f_{11}+f_{22}-e)=0\,.\]
We shall show that $f_{11}+f_{22}=e$. Put $y= f_{11}+f_{22}-e$. We can see that
\[ y\in R(\{h\})\subset R(\{x\})=(\un -e)A\,.\]
So $y\in (\un -e)A\cap eA=\{0\}$. Therefore $f_{11}+f_{22}=e$.
Set $C=AW^\ast(e,f_{11})$. Employing the previous identities   we have that $f_{22},h\in C$. Now (\ref{ctverec}) implies that $AW^\ast(e,f)$ is a $\ast$-subalgebra of $M_2(C)$. \\

     As $e$ can be identified with identity of $C$ and $f_{11}\ge 0$, (in particular $f_{11}$ is self-adjoint) identifications  (\ref{newstar}) follows.
\end{proof}

\begin{pr}\label{okor}
Let $e$ and  $f$ be projections in $A$. Then the algebra $AW^\ast(e,f)$ is contained in a  $AW^\ast$-subalgebra of $A$ isomorphic to \[B\oplus M_2(C)\,,\] where $B$ and $C$ are abelian \AW s. 
\end{pr}

\begin{proof}
Passing to hereditary subalgebra we can assume that $e\lor f=\un$. We set
\[ e_0=e-e\land f - e\land \kol f \qquad e_1=e\land f + e\land \kol f\]
\[ f_0=f-e\land f - \kol e\land  f \qquad f_1=e\land f + \kol e\land  f\,.\]
Then 
\[ \un = e\land f + e\land \kol f + \kol e\land f + \kol e \land \kol f + e_0\lor f_0\,.\]
Let us observe that   $e_0$ and $f_0$ are in position $p$ in the hereditary subalgebra \[(e_0\lor f_0)A(e_0\lor f_0)\,.\] 
The proof is completed by application of \prref{two projections}.
\end{proof}

\begin{theo}\label{isoclinic}
Let $e$ and $f$ be projections in a \AW{}  $A$ with $\nor{e-f}<1$ and $e\land f=0$. Then there is a projection $g$ in $A$ isoclinic to both $e$ and $f$ with the angle
\[ \al= \frac 12\sin^{-1}\nor{e-f}\,.\]
\end{theo}

\begin{proof}
As $\nor{e-f}<1$, we have that $e\land \kol f=\kol e\land f=0$. 
 Moreover, without loss of generality we can assume that $e\lor f=\un$. In that case $e$ and $f$ are in position $p$ and so they can live in an \AW{} isomorphic to
$C(X,M_2(\cc))$ for some compact Hausdorff  space $X$. Moreover, using
\prref{two projections}, 
we can represent $e$ and $f$ as matrix valued functions $e(x)$ and $f(x)$ 
on $X$ such that
\[ e(x)=\left(  \begin{array}{cc}
  1 & 0  \\
     0  & 0
\end{array} 
\right)  \mbox{ and } \qquad f(x)=\left(  \begin{array}{cc}
  a(x) & \sqrt{a(x)-a(x)^2}  \\
 \sqrt{a(x)-a(x)^2}      &1-a(x) 
\end{array}
\right)\,,
\]
where $a(x)$ is a continuous function on $X$ with values in $[0,1]$ 
 Now we can proceed exactly as in the proof of analogous theorem for \vNa s (see surveys \cite[Theorem 5.3.5, p.130]{qmt},\cite{Maeda}).  It is shown there that the desired isoclinic projection $g\in C(X, M_2(\cc))$ is given by the following formula

\[ g(x)= \left(  \begin{array}{cc}
  \l & \sqrt{\l-\l^2}\,\, \omega(x) \\
    \sqrt{\l-\l^2}\,\, \omega(x)   & 1-\l
\end{array}
\right)\,,\]
where $\l$ and $\omega(x)$ are specified as follows: $\l =\cos^2 \al$; 
\[ \omega(x)=e^{i\cos^{-1}\biggl(\sqrt{\frac 1{a(x)}-1}\, \frac{\l-\frac 12}{\l-\l^2}\biggr)}\,,\]
if $a(x)\not=0$; and $\omega(x)=1$ otherwise.

\end{proof}

\begin{lem}\label{halving}
 Given a projection $e$ in $A$ there are projections $p,q,r$ in $A$ \st{}
\begin{enumerate}
\item $p+q+r=e$
\item $p\sim q$
\item $r$ is abelian.
\end{enumerate}
\end{lem}

\begin{proof}
Let us observe that using orthogonal additivity of \AW s \cite[Corollary 1, p.80]{Berberian} and  the fact that  sum of  very orthogonal family of abelian projections is an abelian projection again \cite[Proposition 8,p. 91]{Berberian} the following holds: If  there is a central partition of unity $\sum_\al z_\al =\un$ \st{}  the present lemma holds for every $z_\al A$, then it holds for whole of $A$. Employing decomposition  of $A$ into Types \cite{Berberian}, we can reduce the proof to the case where $A$ is properly infinite, finite Type $II$ or homogeneous finite Type $I$.
If $A$ is properly infinite  or of finite Type $II_1$, then each projection
$e$ can be halved: $e=p+q$, where $p\sim q$, and so the Lemma~\ref{halving} holds with $r=0$.
It remains to prove the statement under  condition that $A$ is of finite Type $I_n$, $n\ge 2$. The hereditary subalgebra $eAe$ is finite again
\cite[Proposition 1, p. 89]{Berberian}. Moreover it remains to be of Type $I$ for the following reason: Suppose, on the contrary that   $eAe$ is not of Type $I$. Then there is a nonzero central  projection $z$ in $eAe$ \st{} $z$ majorizes no nonzero abelian projection. But in Type $I$ algebra every nonzero projection majorizes some
nonzero abelian projection  by \cite[Lemma 1, p. 113]{Berberian}
 Therefore $eAe$ decomposes into finite direct sum of subalgebras of Type $I_k$,
where $k\le n$. (This is due to the fact that Type $I_n$ algebra cannot contain more than $n$ equivalent nonzero abelian   projections.) Working in the hereditary subalgebra $eAe$ we can again pass to its  homogeneous direct summands. Consequently, we can assume without loss of generality that $eAe$ is of Type $I_l$, where $l\in \nn$.  There are  equivalent orthogonal abelian projections $\seq gl$ with sum $e$. If $l$ is even, then we can halve $g$ into $g_1+\cdots +g_{l/2}$ and $g_{l/2+1}+\cdots+g_l$. If $l$ is odd, then we can decompose  $e$ 
into sum of  abelian projection $g_1$ and two equivalent orthogonal projections $g_2+\cdots +g_{\frac{l-1}2+1}$ and $g_{\frac{l-1}2+2}+\cdots +g_{l}$.
This completes the proof.
\end{proof}

\begin{lem}\label{2by2}
Let $A$ has no Type $I_2$ direct summand. Suppose that $B$ is a \Csa{} of $A$
$\ast$-isomorphic to $M_2(\cc)$. Then $B$ is a subalgebra of the direct sum $C\oplus D$, where 
\begin{enumerate}
\item $C$ is either zero or a copy of $M_4(\cc)$,
\item $D$ is either zero or a subalgebra of another algebra that is isomorphic to $M_3(\cc)$. 
\end{enumerate} 
\end{lem}

\begin{proof}
Suppose that $B$ is an $AW^\ast$-subalgebra of $A$ $\ast$-isomorphic to $M_2(\cc)$. Then $B$ is generated by matrix units corresponding to two equivalent nonzero  orthogonal projections (atoms in $B$) $e$ and $f$. By \lemref{halving} we can find projections $e_1,e_2,e_3$ and $f_1,f_2,f_3$ \st{} $e=e_1+e_2+e_3$, $f=f_1+f_2+f_3$, $e_1\sim e_2$, $f_1\sim f_2$, and $e_3$ and $f_3$ are abelian. Then $e_1\sim e_2\sim f_1\sim f_2$ are orthogonal. If they are nonzero then the corresponding matrix unit allows us  to embed these projections into a subalgebra $C$ that is a  copy of $M_4(\cc)$. Let $D$ be a subalgebra generated by   $e_3$ and $f_3$. Suppose that $D$ is nonzero. We shall prove that there is a projection $h\le \un - e_3-f_3$, that is equivalent to $e_3$. 
This will complete the proof.  Let us put $z=c(e_3)=c(f_3)$. As $A$ has no direct summand  of Type $I_2$, we see that $z-e_3-f_3$ is nonzero, for otherwise
$z$ would be a  sum of two equivalent orthogonal abelian projections 
and this would induce  Type $I_2$ direct summand of $A$. For this reason we can work in subalgebra $zA$ in which $e_3$ is a faithful projection. In particular, there is no loss  of generality in assuming that $A$ is of Type $I$.
By considerations analogous to that in the beginning of the proof of  \lemref{halving}, we are able to reduce the argument to the case when  $A$ is homogeneous; that is, of  Type $I_n$, $n\ge 3$. Therefore $A$ contains three equivalent orthogonal faithful abelian projections $h_1\sim h_2\sim  h_3$. All are equivalent to   $e_3$. According to \cite[Proposition 5, p.106]{Berberian} 
if two finite projections in  a \AW{} are equivalent, then the same holds for their complements 
 Having this in mind  and using the fact that $h_1+h_2$ and $e_3+f_3$ are equivalent finite projections, we see that 
$\un-e_3-f_3$ contains an abelian projection equivalent to $h_3$ and therefore to $e_3$ as well. 
\end{proof}

\section{Quasi-linear functionals on \AW s}

\begin{df}
Let $A$ be an \AW. A mapping $\mu:A\to \cc$ is called quasi-linear functional if the following holds
\begin{enumerate}
\item $\mu${} is linear on any abelian $AW^\ast$- subalgebra of $A$. 
\item $\mu(x+iy)=\mu(x)+i\mu(y)$ for all self-adjoint elements $x,y\in A$.
\item $\mu$ is bounded on the unit ball of $A$.
\end{enumerate}
Moreover, we shall call $\mu$ self-adjoint if it takes real values on self-adjoint elements. We shall define a norm of a quasi-linear functional 
$\mu$ by
\[ \nor\mu=\sup\set{|\mu(x)|}{x\in A\,,  \|x\|\le 1}\,.\]
\end{df}
     
\begin{df}
A measure \ro{} on $A$ with values in a normed space $X$ is a bounded map $\mu: P(A)\to X$ satisfying the following condition:
\[ \ro(e+f)=\ro(e)+\ro(f)\]
whenever $e$ and $f$ are orthogonal projections in $A$. 
\end{df}

\begin{pr}\label{3.3}
Every complex measure $\ro{}$ on $A$ extends uniquely to a  quasi-linear functional $\mu$ on $A$. Moreover, if $\ro{}$ is real then $\mu{}$ is self-adjoint. 
\end{pr}

\begin{proof} The proof is the same as in \cite[Proposition 5.2.6, p.125]{qmt}. 
\end{proof}

\begin{pr}\label{3.5}
Suppose that $A$ is an \AW{} for which every quasi-linear functional is linear.
Then any bounded measure \ro{} on $P(A)$ with values in Banach space $X$ extends to a bounded linear operator $T$ from $A$ to $X$. 
\end{pr}

\begin{proof} The proof is the same as the proof of \cite[Theorem 5.2.4, p.123]{qmt}
\end{proof}

  We shall often use the following fact:
  
  \begin{pr}\label{konvex}
Let $A$ be an \AW. For any $0\le x\le \un$ in $A$ there are projections $(e_n)$
lying in the commutative \AW{} generated by $x$ \st{}
\[ x= \sum_{n=1}^\8 \frac 1{2^n}e_n\,.\]
\end{pr}

\begin{proof}
It holds for any \Ca{} enjoying the spectral axiom (see \cite[p.367]{Str-Sz}). 
\end{proof}

  
 



  \begin{pr}\label{local} 
Let $A$ be an \AW{} without Type $I_2$ direct summand. Then every quasi linear functional is linear on  every subalgebra of $A$ isomorphic to $M_2(\cc)$. 
\end{pr}

\begin{proof}
Any copy $B$ of $M_2(\cc)$ is embedded into matricial subalgebra of $A$ for which the classical \GT{} holds
 \cite{Gleason}. Therefore $\mu$ is linear on $B$. 
\end{proof}

  \begin{pr}\label{Lipschitz}
Any quasi linear functional  $\mu$ on an \AW{} $A$, that has no Type $I_2$ direct summand,
is Lipschitz on $P(A)$.
\end{pr}
  
  \begin{proof}
The proof is the same as in Theorem 5.3.8 in \cite{qmt}, we present it here  for the sake of completeness.  

 As $\|e-f\|\le 1$ for all projections $e$ and $f$, we can verify Lipschitz condition for a pair of projections with $\nor{e-f}<1$. Moreover, we can (by discarding $e\land f$) suppose that $e\land f=0$. Employing \thmref{isoclinic},
 we can find a projection $h$ isoclinic to both $e$ and $f$ with angle 
 $\al=\frac 12\sin^{-1}\nor{e-f}<\frac{\pi}4$. As we know, $C^\ast(e,h)$ and $C^\ast(f,h)$ are isomorphic to $M_2(\cc)$. Therefore $\mu{}$ is linear on these algebras (\prref{local}). In particular,
     \[ |\mu(e)-\mu(h)| \le \|\mu\| \|e-h\|=\nor{\mu}\sin\biggl(\frac 12\sin^{-1}\nor{e-f}\biggr)\le \nor{\mu}\cdot \nor{e-f}\,.\] 
Similarly,     
          \[ |\mu(f)-\mu(h)| \le \|\mu\| \|e-f\|\,.\]
          This gives
\[ |\mu(e)-\mu(f)|\le 2 \nor{\mu}\nor{e-f}\,.\]

\end{proof}

We shall need some notation. 
 Suppose that $X$ is a Stonean space and consider \AW{} $A=C(X, M_n(\cc))\simeq C(X)\otimes M_n(\cc)$. An  element $f\in A$ is called locally constant if it attains finitely many values. In other words, $f$ is locally constant \ifff{}    there is a partition \seq O k{}  of $X$ consisting of clopen sets \st{} $f$ is constant on each $O_i$. It is clear that the set $B$ of all locally constant functions is a  $\ast$-subalgebra of $A$. Moreover, $B$ is dense in $A$.
It follows directly from the fact that in  any \AW{} the set of self-adjoint elements with finite spectrum is dense in self-adjoint part of the algebra.
(Alternatively, it can be established by topological considerations.)


\begin{theo}\label{typeIn}
Let $A$ be of Type $I_n$, $n\ge 3$. Then any quasi linear functional on $A$ is linear. 
\end{theo}

\begin{proof}
Let $\mu$ be a quasi linear functional on $A$. It is enough to assume that 
$\mu$ is self-adjoint. By the structure theory of \AW s $A$ can be identified with $C(X,M_n(\cc))$, where $X$ is a Stone space. Let $B$ be a subalgebra of $A$ consisting of locally constant functions. First we show that $\mu$ is linear on  $B$. For this take $a,b\in B$. We can find a partition of $X$, consisting of clopen sets \seq Ok{}, \st{} both $a$ and $b$ are constant 
on each $O_i$. The set of all locally constant functions with this property forms a $\ast$ subalgebra $C$  of $B$ that can be identified with $k$-fold direct sum of the matrix algebras $M_n(\cc)$. By the classical \GT{}, $\mu$ is linear on $C$. Since $a,b\in C$ we can see that $\mu(a+b)=\mu(a)+\mu(b)$. This way we have established linearity of $\mu$ on $B$. Therefore, there is a unique bounded linear extension,  $\ro$, of $\mu|B$ to $A$ since $B$ is dense in $A$. It is clear  that $\ro$ coincides with  $\mu$ on $P(B)$. However, $\mu$ is uniformly continuous on $P(A)$ by \prref{Lipschitz} and $P(B)$ is dense in $P(A)$. Therefore,
$\ro$ and $\mu$ coincide on $P(A)$. However every quasi linear functional on
$A$ is uniquely determined by its values on projections. Therefore $\mu=\ro$
and the proof is completed. 
\end{proof}

\begin{cor}\label{two}
Let $A$ be an \AW{} without Type $I_2$ direct summand and let $\mu$ be a quasi linear functional on $A$.  Let $e,f$ be projections in $A$. Then $\mu$ is linear on subalgebra $AW^\ast(\un,e,f)$. 
\end{cor}

\begin{proof}
First we show that $\mu$ is linear on $AW^\ast(e,f)$.  By  \prref{okor}  there is no loss in assuming that $AW^\ast(e,f)$ is a direct sum of abelian algebra
and an algebra $D$ that can be identified with $C(X,M_2(\cc))$, where $X$ is a Stonean space. It is enough to establish linearity of $\mu$  on the latter direct summand. For this we can proceed verbatim like in the proof of \thmref{typeIn}. 
Consider a subalgebra $B\subset D$  of locally constant functions in $C(X, M_2(\cc))$. Given two   elements $a$ and $b$ of $B$ we can embed them into the direct sum of copies of $M_2(\cc)$. By \prref{local}, $\mu$ is linear on this 
subalgebra. It implies  that $\mu(a+b)=\mu(a)+\mu(b)$. Then rest  is  the same as in  the proof of \thmref{typeIn}. \par
When $\mu$ is linear on $AW^\ast(e,f)$, then it must be linear on $AW^\ast(\un, e,f)$ since $AW^\ast(\un, e,f)=\cc \un +  AW^\ast(e,f)$.
\end{proof}
 

\section{\DT{} for \AW s}

   The next proposition  can be proved in the same way as 
   in \cite[Theorem 8.1.1, p. 255]{qmt} for von Neumann algebras. In fact it holds not only for  \AW s  but for all \Ca s of real rank zero. 
   
   \begin{pr}\label{Jordan}
Any bounded linear map $\Phi: A\to B$ between \AW s  that preserves   projections (that is, $\Phi(P(A))\subset P(B)$) is a Jordan $\ast$-homomorphism.  
\end{pr}

 Following terminology in \cite{Active} we say that a map between projection lattices of \AW s  is a {\bf COrtho}-morphism if it preserves orthocomplementations and suprema of arbitrary projections. Then it preserves order, unit, and orthogonality of projections. 
 
 We now proceed to the main theorem.
 
 \begin{theo}\label{main}
Let $A$ be an \AW{} without Type $I_2$ direct summand and $B$ be an \AW. Let $\f:P(A)\to P(B)$ be a  {\bf COrtho}-morphism. Then $\f$ extends to a Jordan $\ast$-homomorphism $\Phi: A\to B$. 
\end{theo}
   
    \begin{proof}
First we show that \f{} induces a bounded measure on $P(A)$ with values in $P(B)$.
Indeed, let us take two orthogonal projections $e$ and $f$ in $A$. Then $\f(e)$ and $\f(f)$ are orthogonal. As \f{} preserves suprema we have
\[ \f(e+f)= \f(e\lor f)= \f(e)\lor \f(f)=\f(e)+\f(f)   \,.\]
Therefore \f{} is a bounded measure on $P(A)$. Let us take two projections $e$ and $f$ in $P(A)$. By \corref{two}  every quasi linear functional on $G=AW^\ast(\un,e,f)$ is linear and so, by \prref{3.5}, the restriction of $\f$ to this  algebra extends to a bounded  operator, say $T$, from $G$ into $B$.
By \prref{Jordan},  $T$ is a Jordan $\ast$-homomorphism. As any Jordan $\ast$-homomorphism preserves triple products we have 
\[ \f((\un - 2e)f(\un -2e))= T((\un - 2e)f(\un -2e))=\]\[ =(\un-2T(e))T(f)(\un-2T(e))
=(\un-2\f(e))\f(f)\f(\un-2\f(e))\,.\]
But, according  to deep result of Heunen and Reyes \cite[Theorem 4.6]{Active} 
the above equality is equivalent to the fact that  $\f$ extends to a Jordan $\ast$-homomorphism between $A$ and $B$. 
\end{proof}

Next we  extend \DT{} to $AW^\ast$-algebras. 
A map $\f: P(A)\to P(B)$ between projection lattices of \AW s  $A$ and $B$  
is called orthoisomorphism if it is a bijection preserving orthogonality in both directions, in the sense that $ef=0$ \ifff{} $\f(e)\f(f)=0$. Let us remark that every orthoisomorphism is a {\bf COrtho}-morphisms.\\

The following theorem generalizes the \DT.\\

\begin{theo}\label{Dye}
Let $\f:P(A)\to P(B)$ be an orthoisomorphism between projections lattices of \AW s $A$ and $B$, where $A$ has no Type $I_2$ direct summand. Then there is a Jordan $\ast$-isomorphism $\Phi: A\to B$ that extends $\f$. 
\end{theo}   
   
   \begin{proof}
 As \f{} is a {\bf COrtho}-morphism,   we have by  \thmref{main} that there is a Jordan $\ast$-homomorphism $\Phi$ from $A$ to $B$ extending \f{}. It remains to show that $\Phi$ is an isomorphism. Every Jordan $\ast$-homomorphism has a closed range. This, together with the fact that image of $\Phi$ contains $P(B)$ that generates $B$ as a Banach space, we infer that $\Phi$ is surjective. Now we establish  the injectivity of $\Phi$. Using the fact that 
kernel of $\Phi$ is a Jordan $\ast$-ideal generated linearly by  positive elements, we can see that for proving injectivity it suffices to show that
$\Phi$ is nonzero on every element $0\le x \le \un$.  It follows  from \prref{konvex} that any such $x$ dominates a nonzero positive  multiple $\l e$ of
a projection $e$. Then $\Phi(\l e)=\l\Phi(e)$ is nonzero by the hypothesis. As $\Phi$ preserves the order, we obtain that $\Phi(x)$ must be nonzero. 
 \end{proof}
   
\begin{df}   
By a quasi Jordan $\ast$-homomorphism  between \AW s $A$ and $B$ we mean a map
$\Phi: A\to B$ that satisfies the following conditions:
\begin{enumerate}
\item $\Phi$ preserves the $\ast$ operation; that is, $\Phi(a^\ast)=\Phi(a)^\ast$ 
for all $a\in A$.
\item $\Phi$ is a Jordan $\ast$-homomorphism on every abelian $AW^\ast$-subalgebra $C$ of $A$. That is, $\Phi$ is linear on $C$ and $\Phi(a^2)=\Phi(a)^2\,$ for every $a\in C$.
\item $\Phi(a+ib)=\Phi(a)+i\Phi(b)$ for all self-adjoint $a,b\in A$.
\end{enumerate}
Moreover we shall call a quasi Jordan $\ast$ homomorphism normal if it preserves
increasing nets of projections, that is if $\Phi(e_\al)\nearrow \Phi(e)$ in $P(B)$ whenever $e_\al\nearrow e$ in $P(A)$. Finally, by a quasi Jordan $\ast$-isomorphism we understand a quasi  Jordan $\ast$-homomorphism that is a bijection and  whose inverse is again a   quasi  Jordan $\ast$-homomorphism. 
\end{df}

\begin{theo}\label{linJordan}
Any normal quasi Jordan $\ast$-homomorphism $\Phi: A\to B$, where $A$ and $B$ are \AW s, $A$ not having Type $I_2$ direct summand, is a Jordan $\ast$-homomorphism.  
Moreover, any quasi Jordan $\ast$-isomorphism $\Phi: A\to B$ is a Jordan $\ast$-isomorphism.
\end{theo}

\begin{proof}
First we show that $\Phi$ restricts to a {\bf Cortho}-morphism $\f$ between projection lattices.
 For this it is enough to establish that $\Phi$ preservers suprema of two elements (see e.g. \cite[Lemma 3.2]{Active}). Let 
us take projections $e$ and $f$ in $A$. As $\Phi$  preserves order we have that $\Phi(e)\lor \Phi(f)\le \Phi(e\lor f)$. 
The sum $e+f$ is a self-adjoint element and so $AW^\ast(e+f)$ is abelian algebra  isomorphic to some $C(X)$, where $X$ is Stonean. 
From this one can deduce that there is an increasing sequence $(h_n)$ of projections in $AW^\ast (e+f)$ \st{} $h_n\nearrow RP(e+f)=e\lor f$ and $e+f\ge \frac 1n\, h_n$ for each $n$.
 Then $\Phi(h_n)\nearrow \Phi(e\lor f)$. Observe further that working in \AW{} generated by two projections $e$ and $f$ we have linearity of $\Phi$ on this subalgebra (see the proof of Theorem~\ref{main})  and so we have  that $
\Phi(e+f)=\Phi(e)+\Phi(f)$.  Therefore,  $\Phi(e+f)=\Phi(e)+\Phi(f)\ge \frac 1n \Phi(h_n)$. It implies that 
       $\Phi(h_n)\le RP(\Phi(e)+\Phi(f))=\Phi(e)\lor \Phi(f)$. Therefore $\Phi(e\lor f)\le \Phi(e)\lor \Phi(f)$, giving the reverse inequality. 
 By \thmref{main},
$\f$ extends to a bounded linear map $\Psi$ from $A$ to $B$. Since $\Phi$ and $\Psi$ coincide on
$P(A)$,  they have to be equal. \\

To prove the second statement, let us observe first that $\f$ is  an orthoisomorphism. As $\f$ is injective it will suffice to show that $\Phi(P(A))=P(B)$. To this end consider a projection $p$ in $P(B)$. There is a self-adjoint element $x\in A$ \st{} $\Phi(x)=p.$  Then $\Phi(x)=\Phi(x)^2=\Phi(x^2)$, giving by injectivity of $\Phi$, that  $x=x^2$. Therefore $x$ is a projection. 
By  Theorem~\ref{Dye} $\f$ extends to a Jordan $\ast$-isomorphism $\Psi$ 
between $A$ and $B$. As above $\Phi=\Psi$. 
 \end{proof}

In  conclusion of this paper, we apply our main result to show that structure of abelian \Csa s of a \AW{} algebra determines  Jordan structure of $A$.

Let $Abel(A)$  be  a set of all  abelian \Csa{}  of a unital \Ca{} $A$ that contain the unit of $A$.
 When ordered by inclusion, we obtain the posets that play an important role in foundations of physics \cite{Deep}. 
The following application of our main results allows one to identify order isomorphisms of the structure of abelian subalgebras with  Jordan $\ast$-isomorphisms. Let us recall that an order isomorphism between two posets is a bijection preserving the order in both directions. \\

\begin{theo}
Let $A$ be an \AW{} without Type $I_2$ direct summand and $B$ be another \AW.  Then for any 
order isomorphism $\f: Abel(A)\to Abel(B)$ 
there is a unique Jordan $\ast$-isomorphism $\Phi: A\to B$  \st{}
     \[ \Phi(C)=\f(C)\]
     for all $C\in Abel(A)$. 
\end{theo}

\begin{proof}
It was shown in \cite{Ham-Abel} that any order isomorphism  between $Abel(A)$ and $Abel(B)$ is implemented in the above sense by a quasi Jordan $\ast$-isomorphism between $A$ and $B$.
 The result then follows from \thmref{linJordan}. 

\end{proof}

{\bf Acknowledgment:}  This work was supported by  the ``Grant Agency of the Czech
Republic",
  Grant No. P201/12/0290  ``Topological and geometrical properties of Banach spaces and operator algebras".

\end{document}